\documentclass[12pt]{amsart}
\usepackage{a4wide}
\usepackage[utf8]{inputenc}
\usepackage{tikz}
\usepackage[backend=biber,style=alphabetic,maxbibnames=99]{biblatex}  
\usepackage{multirow}
\usepackage{textcomp}
\usepackage{csquotes} 
\usepackage{listings}
\usepackage{hyperref}
\hypersetup{colorlinks=false}
\usepackage{amssymb}
\usepackage{enumitem}
\usepackage[none]{hyphenat}

\usepackage{amsmath,amssymb}
\newcommand{\vertiii}[1]{{\left\vert\kern-0.25ex\left\vert\kern-0.25ex\left\vert #1 
    \right\vert\kern-0.25ex\right\vert\kern-0.25ex\right\vert}}

\newtheorem{theorem}{Theorem}[section]
\newtheorem{thmx}{Theorem}

\newtheorem{lemma}[theorem]{Lemma}
\newtheorem{claim}{Claim}
\newtheorem{question}{Question}

\theoremstyle{definition}

\newtheorem{fact}{Fact}

\author{Andrés Quilis}
\address{Universitat Polit\`ecnica de Val\`encia. Instituto Universitario de Matem\'atica Pura y Aplicada, Camino de Vera, s/n
46022 Valencia (Spain); and Czech Technical University in Prague, Faulty of Electrical Engineering. Department of Mathematics, Technick\'a 2, 166 27 Praha 6 (Czech Republic)}
\email{anquisan@posgrado.upv.es}
\title[Renormings preserving local geometry]{Renormings preserving local geometry at countably many points in spheres of Banach spaces and applications}
\begin{document}
\begin{abstract} 
We develop tools to produce equivalent norms with specific local geometry around infinitely many points in the sphere of a Banach space via an inductive procedure. We combine this process with smoothness results and techniques to solve two open problems posed in the recently published monograph \cite{GuiMonZiz22} by A. J. Guirao, V. Montesinos and V. Zizler. Specifically, on the one hand we construct in every separable Banach space admitting a $C^k$-smooth norm an equivalent norm which is $C^k$-smooth but fails to be uniformly Gâteaux in any direction; and on the other hand we produce in $c_0(\Gamma)$ for any infinite $\Gamma$ a $C^\infty$-smooth norm whose ball is dentable but whose sphere lacks any extreme points. 
\end{abstract}
\maketitle

\section{Introduction}

In an infinite-dimensional Banach space, the existence of equivalent norms with different properties related to differentiability and convexity has a profound impact on the structure and geometry of the space and its dual. This has lead to Renormings in Banach spaces being a very active and interesting topic of research in Banach space theory. We refer the reader to \cite{DevGodZiz93,God01} and to the recent monograph \cite{GuiMonZiz22} for a deep study of this subject, and as a sample of the wealth of results in it. For a selection of relatively recent articles dealing with concepts close the topics we discuss in this note we point to \cite{Ryc05,JohRyc07,HajJoh09,HajMonZiz12,HajTal14,HajRus17}.

Differentiability and convexity qualities of a given norm in a Banach space are related to the geometry of its unit sphere. Among these properties, some relate to the local shape around a specific point in the sphere, while others ask for certain uniformity on the entire set. As a central example to this note, notice that to check that the norm is Fréchet differentiable, we only have to look at a small enough neighborhood of each point in the sphere individually. On the other hand, to obtain that the norm is uniformly Gâteaux in a particular direction, the size of the neighborhood that we check cannot depend on the point of the sphere we are considering, but only on the quality of the approximation to the derivative. 

There are many examples of norms in Banach spaces satisfying strong local properties but failing their uniform counterpart. A well known example is the fact that a Banach space admitting a uniformly rotund norm is necessarily superreflexive (see e.g.: Theorem IV.4.1 in \cite{DevGodZiz93}), while the class of Banach spaces admitting a locally uniformly rotund norm is considerably larger (see e.g.: Chapter VII in \cite{DevGodZiz93}), including for instance all weakly Lindelöf determined spaces or all duals of Asplund spaces.

More relevant to our discussion is the result proven in \cite{BorFab93b} by Borwein and Fabián, which states that in every separable Banach space there exists an equivalent norm which is uniformly Gâteaux but is not everywhere Fréchet differentiable.

In this note we develop a process to construct norms in Banach spaces satisfying strong local properties of smoothness, which fail to be uniformly Gâteaux differentiable in any direction. In the special case of $c_0(\Gamma)$, the same process allows us to obtain a $C^\infty$-smooth norm locally dependent on finitely many coordinates (LFC), which moreover satisfies the somewhat opposite property of having a dentable unit ball.

Briefly and intuitively, the strategy used to produce our renormings consists in defining countably many equivalent norms in a Banach space, all of which satisfy the desired local properties of smoothness but with decreasing quality around respective specific points. The final step of the process consists in combining all of these renormings into a single one which preserves the local geometry of each individual renorming at their respective critical region. The idea to make this combination possible is to define the sequence of individual renormings using an almost biorthogonal sequence of pairs $\{(x_n,f_n)\}\subset S_X\times S_{X^*}$, where moreover $\{f_n\}_{n\in\mathbb{N}}$ is weak$^*$ null. This will ensure that the different renormings do not overlap each other excessively in their critical regions and that the final renorming is indeed equivalent to the original one. 

Let us now state the two specific applications that we obtain:

\begin{thmx}
\label{Main_Theorem_A}    
Let $X$ be a separable Banach space with a $C^k$-smooth norm with $k\in\mathbb{N}\cup\{\infty\}$. Then $X$ admits an equivalent $C^k$-smooth norm which is not uniformly Gâteaux in any direction. Moreover, the unit ball in this norm is dentable. 
\end{thmx}

\begin{thmx}
\label{Main_Theorem_B}
For every infinite set $\Gamma$, the space $c_0(\Gamma)$ admits an equivalent norm which is $C^{\infty}$-smooth and LFC whose unit ball is dentable. Moreover, this norm is not uniformly Gâteaux in any direction.
\end{thmx}
The motivation for these applications comes from the open problems chapter of the monograph \cite{GuiMonZiz22}. Specifically, the two previous theorems answer the two following questions respectively:

\begin{question}[Page 496, 52.1.2.1 in \cite{GuiMonZiz22}]
Can every infinite-dimensional separable reflexive space be renormed by a Fréchet differentiable norm which is not uniformly Gâteaux?
\end{question}

Indeed, it is well known that every separable reflexive space admits a Fréchet smooth (and thus $C^1$-smooth) renorming (see e.g. \cite{DevGodZiz93}), so Theorem \ref{Main_Theorem_A} solves this question in particular.

\begin{question}[Page 497, 52.1.3.3 in \cite{GuiMonZiz22}]
    Is there a norm on $c_0$ which is $C^\infty$-smooth, and whose ball is dentable but whose sphere has no extreme points?
\end{question}

This second question is answered by Theorem \ref{Main_Theorem_B} since the sphere of an LFC norm in an infinite-dimensional Banach space cannot have any extreme points.

The authors of \cite{GuiMonZiz22} also include a weaker version for each one of these two questions: namely, they ask whether every infinite-dimensional separable superreflexive space can be renormed by a Fréchet smooth norm which is not uniformly Fréchet (Page 496, 52.1.2.4); and if $c_0$ can be renormed with a $C^2$-smooth norm whose ball is dentable (Page 498, 52.1.4.2). Hence, both of these problems are also now solved. 

Let us finish the introduction by explaining the structure of this note. The content is divided into three sections: Section 2 is dedicated to fixing the notation and giving preliminary definitions. Section 3 contains the construction of the basic renorming from which we will obtain the final result. We study this renorming by giving quantitative estimates on the speed of convergence for the derivative of the norm at specific points, and on the diameter of certain slices of the unit ball. Finally, in section 4 we use known and deep results in Banach space theory and smoothness to develop the process to combine the individual renormings in the previous section into one final renorming in a suitable way.

\section{Notation and preliminaries}
We introduce some basic notation and results. Some more definitions will be presented in the main sections when needed.

All Banach spaces in this note are assumed to be real Banach spaces. A Banach space $X$ endowed with a norm $\|\cdot\|$ will be denoted by $(X,\|\cdot\|)$, and $\|\cdot\|$ will always denote the starting norm, as opposed to the renormings of $\|\cdot\|$ defined on the space $X$, denoted always by the symbol $\vertiii{\cdot}$ with possibly an appropriate subindex. In the dual space $X^*$, the dual norm induced by an equivalent norm $\vertiii{\cdot}$ will be denoted by $\vertiii{\cdot}^*$.

In a Banach space $(X,\|\cdot\|)$, given an equivalent norm $\vertiii{\cdot}$ and for any point $x\in X$ and $r>0$, the ball centered at $x$ of radius $r$ is denoted by $B_{\vertiii{\cdot}}(x,r)$. The unit ball of a space $(X,\|\cdot\|)$ in an equivalent norm $\vertiii{\cdot}$ is denoted by $B_{(X,\vertiii{\cdot})}$, while the unit sphere will be written as $S_{(X,\vertiii{\cdot})}$. We may simply write $B_{\vertiii{\cdot}}$ and $S_{\vertiii{\cdot}}$ wherever no ambiguity is possible. 

For an infinite set $\Gamma$, we consider the Banach space $c_0(\Gamma)$ to be the space formed by all real valued functions defined on $\Gamma$ such that for every $x\in c_0(\Gamma)$, the set $\{\gamma\in \Gamma\colon x(\gamma)>\varepsilon\}$ is finite for all $\varepsilon>0$. We will consider the usual supremum norm as the starting norm in $c_0(\Gamma)$, denoted by $\|\cdot\|_{\infty}$. For every $\gamma\in \Gamma$, we will denote by $e_\gamma$ the vector in $c_0(\Gamma)$ such that $e_\gamma(\gamma)=1$ and $e_\gamma(\xi)=0$ for all $\xi\in\Gamma\setminus\{\gamma\}$.

We will use the usual definitions of \emph{(uniform) Gâteaux, Fréchet and 
$C^k$-smoothness for $k\in\mathbb{N}\cup\{\infty\}$}. For a deep study of differentiability in Banach spaces we direct the reader to the monographs \cite{DevGodZiz93,HajMic14,GuiMonZiz22}. We briefly outline the characterizations of smoothness concepts that we will use in this note while introducing some more notation:

Given a norm $\vertiii{\cdot}$ in a Banach space $X$, a point $x\in X\setminus \{0\}$, a direction $h\in X\setminus\{0\}$ and $t\neq 0$, we define the quotient:
$$D(\vertiii{\cdot},x,h,t)=\frac{\vertiii{x+th}+\vertiii{x-th}-2\vertiii{x}}{t}. $$
The norm $\vertiii{\cdot}$ is Gâteaux differentiable at $x$ in the direction of $h$ if and only if $D(\vertiii{\cdot},x,h,t)$ converges to $0$ as $t$ goes to $0$. The norm $\vertiii{\cdot}$ is Fréchet differentiable at $x$ if and only if the previous limit converges to $0$ as $t$ goes to $0$ uniformly with respect to $h$, when $h$ ranges in the sphere of an equivalent norm in $X$. We say that $\vertiii{\cdot}$ is Gâteaux (resp. Fréchet) if it is Gâteaux (resp. Fréchet) at every point in $X\setminus\{0\}$. 

The norm $\vertiii{\cdot}$ is uniformly Gâteaux in the direction of $h\in X\setminus \{0\}$ if and only if $D(\vertiii{\cdot},x,h,t)$ converges to $0$ as $t$ goes to $0$, and this limit is uniform in $x$ when $x$ ranges in the sphere of an equivalent norm in $X$. Notice that $\vertiii{\cdot}$ is uniformly Gâteaux in the direction of $h\in X\setminus\{0\}$ if and only if it is uniformly Gâteaux in the direction of $\lambda h$ for all $\lambda\neq 0$.

We finish the section by recalling some more definitions: 

If $U$ is an open subset of a Banach space $X$, and $f\colon U\rightarrow \mathbb{R}$ is an arbitrary function, we say that $f$ \emph{locally depends on finitely many coordinates} (is LFC, for short) at a point $x\in U$, if there exists an open neighborhood $W\subset U$ of $x$ and a finite set of functionals $\{f_1,\dots,f_n\}\in X^*$ such that $f(y)=f(x)$ if $y\in W$ with $\langle f_i,x\rangle=\langle f_i,x\rangle$ for all $i=1,\dots n$. We say that $f$ is LFC if it is LFC at every point in its domain.

We say that a norm $\vertiii{\cdot}$ in a Banach space is LFC if it is LFC on $X\setminus \{0\}$. 

Given a convex set $C$ of a Banach space $X$, we say that a point $x\in C$ is an \emph{extreme point of $C$} if for every pair of points $y,z\in C$ such that $x=\frac{y+z}{2}$ we have that $x=y=z$. Observe that if $\vertiii{\cdot}$ is an LFC norm in an infinite-dimensional Banach space, then no point in the unit ball $B_{\vertiii{\cdot}}$ is an extreme point. 

In a Banach space $X$ with the norm $\vertiii{\cdot}$, given a convex set $C\subset X$, a functional $f\in X^*$ and $a>0$, the \emph{slice defined by $f$ and $a$ of $C$} is the set $S(C,f,a)=\{x\in C\colon \langle f,x\rangle>a\}$. The diameter of a subset $C$ of $X$ in an equivalent norm $\vertiii{\cdot}$ will be denoted by $\text{diam}_{\vertiii{\cdot}}(C)$.

We say that a convex set $C$ is \emph{dentable} if it admits non-empty slices of arbitrarily small diameter in an equivalent norm. 

\section{A renorming with suitable local properties}

In this section we define an equivalent norm in any Banach space $(X,\|\cdot\|)$ associated to a pair $(x_0,f_0)\in S_{\|\cdot\|}\times S_{\|\cdot\|^*}$ such that $\langle f_0,x_0\rangle =1$ and to a positive number $0<\delta<1/2$. After defining this norm we will obtain some differentiability and dentability estimates.

Fix a Banach space $(X,\|\cdot\|)$. For each equivalent norm $\vertiii{\cdot}$ in $X$, each $f\in S_{\vertiii{\cdot}^*}$ and each $0<r<1$, define the sets:
\begin{align*}
        C^+(f,r,\vertiii{\cdot})&=\{y\in X\colon \langle f,y\rangle \geq r\vertiii{y}\},~\text{and}\\
        C^-(f,r,\vertiii{\cdot})&=\{y\in X\colon \langle f,y\rangle \leq -r\vertiii{y}\},
\end{align*}
which are closed convex cones in $X$. We define as well the cone $C(f,r,\vertiii{\cdot})=C^+(f,r,\vertiii{\cdot})\cup C^-(f,r,\vertiii{\cdot})$. We will write $C^+(f,r)$, $C^-(f,r)$ and $C(f,r)$ to denote the cones in the starting norm $\|\cdot\|$. Let us state a simple fact about some points in the complement of these cones that will be useful in the last section:

\begin{fact}
\label{Fact1_Neighbourhood_outside}
Let $(X,\|\cdot\|)$ be a Banach space, let $f_0\in S_{\|\cdot\|^*}$ and let $0<\delta<\frac{1}{2}$. If $x\in X$ satisfies that $|\langle x,f_0\rangle| <\frac{\|x\|}{8}$, then $y\notin C(f_0,1-\delta)$ for all $y\in B_{\|\cdot\|}\left(x,\frac{\|x\|}{8}\right)$. 
\end{fact}
\begin{proof}
    The proof of this fact is elementary.
\end{proof}

\subsection{Definition of the norm}

We proceed to define the unit ball of the desired norm. Consider a pair $(x_0,f_0)\in S_{\|\cdot\|}\times S_{\|\cdot\|^*}$ with $\langle f_0,x_0\rangle =1$ and fix $0<\delta<1/2$ for the rest of the section. Consider the set:
\begin{align*}
    B_{(x_0,f_0),\delta}=\overline{\text{conv}}\left( \left(B_{\|\cdot\|}\setminus C\left(f_0,1-\delta\right)\right)\bigcup\left\{\left(1-\frac{\delta}{2}\right)x_0,-\left(1-\frac{\delta}{2}\right)x_0\right\}\right),
\end{align*}
which is a closed absolutely convex set contained in the unit ball $B_{\|\cdot\|}$. We are going to show that this set is the unit ball of an equivalent norm in $X$. To this end, let us define an auxiliary function which will be used throughout the rest of the section:

For $y\in X$, notice that there exists $t_0\geq 0$ such that 
$$ (1-\delta)\|y+(t-\langle f_0,y\rangle)x_0\| \leq t $$
for all $t\geq t_0$. Since $(1-\delta)\|y-\langle f_0,y\rangle x_0\|\geq 0$, by continuity we obtain that there exists a non-negative scalar $\psi(y)$ so that 
$$ \psi(y)=(1-\delta)\|y+(\psi(y)-\langle f_0,y\rangle)x_0\|. $$
Moreover, we have that if $ (1-\delta)\|y+(t-\langle f_0,y\rangle)x_0\| \leq t $ for some $t\geq 0$ and $\varepsilon>0$, then $(1-\delta)\|y+(t+\varepsilon-\langle f_0,y\rangle)x_0\| < t+\varepsilon$, which implies that $\psi(y)$ is the unique non-negative number satisfying its defining equality for every $y\in X$. 

We can thus consider the well defined function $\psi\colon X\rightarrow [0,\infty)$ as the map such that $\psi(y)$ is the unique non-negative number satisfying $\psi(y)=(1-\delta)\|y+(\psi(y)-\langle f_0,y\rangle)x_0\|$. Moreover, a similar reasoning provides that $\psi$ is a subadditive function. We have as well that $\psi$ is positively homogeneous, i.e.: for every $\lambda\geq 0$ and every $y\in X$ it holds that $\psi(\lambda y)=\lambda\psi(y)$. 

Additionally, it follows from the definition that for every $y\in X$ and every $\lambda\in \mathbb{R}$, we have that $\psi(y+\lambda x_0)=\psi(y)$; in other words, $\psi$ is invariant by translations in the direction of $x_0$. Importantly, we also obtain that $\psi(y)=0$ if and only if $y=\langle f_0,y\rangle x_0$ for every $y\in X$. 

Figure \ref{Fig1_psi} shows the geometric intuition behind the function $\psi$. 

\begin{figure}
    \label{Fig1_psi}
    \centering
    \includegraphics[scale=0.5]{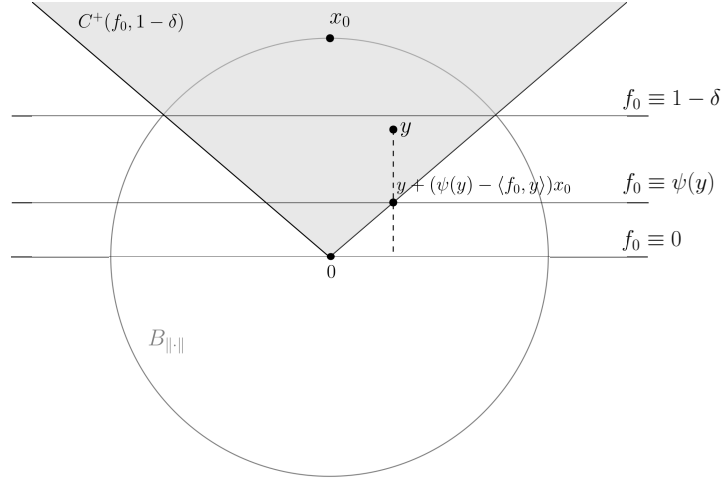}
    \caption{The function $\psi$. The shaded area is the positive cone $C^+(f_0,1-\delta)$.}
\end{figure}

With the definition of the function $\psi$ fixed for the rest of the section, and with its basic properties we have discussed, we can now find a better description of the norm we are constructing:

\begin{lemma}
\label{Prop1_Description_norm}
The set $B_{(x_0,f_0),\delta}$ as defined above is the unit ball of the norm $\vertiii{\cdot}_{(x_0,f_0),\delta}$ defined by:
$$ 
    \vertiii{y}_{(x_0,f_0),\delta}=
    \begin{cases}
        \|y\|,~ &\text{if }y\notin C\left(f_0,1-\delta\right),\\
        \frac{\psi(y)}{1-\delta}+\frac{\langle f_0,y\rangle-\psi(y)}{1-\frac{\delta}{2}},&\text{if }y\in C^+(f_0,1-\delta),\\
        \frac{\psi(-y)}{1-\delta}+\frac{\langle f_0,-y\rangle-\psi(-y)}{1-\frac{\delta}{2}},&\text{if }y\in C^-(f_0,1-\delta),
    \end{cases}
$$
which is an equivalent norm in $(X,\|\cdot\|)$ with
$$ \|\cdot\|\leq \vertiii{\cdot}_{(x_0,f_0),\delta}\leq \frac{1}{1-\delta}\|\cdot\|.$$
\end{lemma}
\begin{proof}
    Since $\psi(0)=0$, the definition of $\vertiii{\cdot}_{(x_0,f_0),\delta}$ is not problematic at this point and $\vertiii{0}_{(x_0,f_0),\delta}=0$. Notice as well that $\vertiii{y}_{(x_0,f_0),\delta}$ is non-negative for all $y\in X$ because $\psi(y)\leq\langle f_0,y\rangle$ for all $y\in C^+(f_0,1-\delta)$.
    
    We start by proving that $\vertiii{y}_{(x_0,f_0),\delta}=0$ if and only if $y=0$. We have already shown that $\vertiii{0}_{(x_0,f_0),\delta}=0$. Suppose now that $\vertiii{y}_{(x_0,f_0),\delta}=0$ for some $y\in X$. If $y\notin C\left(f_0,1-\delta\right)$ a contradiction follows from the fact that $\|\cdot\|$ is a norm. On the other hand, if $y\in C^+(f_0,1-\delta)$, then the inequality $\psi(y)\leq\langle f_0,y\rangle$ implies both that $\psi(y)=0$ and $\langle f_0,y\rangle =0$. The first condition means that $y= \langle f_0,y\rangle x_0$ and with the second we obtain that $y=0$. The final possibility when $y\in C^-(f_0,1-\delta)$ follows from the previous case.

    Since $\psi$ is positively homogeneous and  $-\lambda C^+(f_0,1-\delta)\subset C^-(f_0,1-\delta)$ and $-\lambda C^-(f_0,1-\delta)\subset C^+(f_0,1-\delta)$ for every $\lambda\geq 0$, we have that $\vertiii{\cdot}_{(x_0,f_0),\delta}$ is homogeneous.

    To finish proving that $\vertiii{\cdot}_{(x_0,f_0),\delta}$ is a norm, it only remains to show that it satisfies the triangle inequality. We will do this by proving that $B_{(x_0,f_0),\delta}=\{y\in X\colon \vertiii{y}_{(x_0,f_0),\delta}\leq 1\}$, and the triangle inequality will follow by a straightforward convexity argument. Moreover, this shows as well that $B_{(x_0,f_0),\delta}$ is the unit ball of the norm $\vertiii{\cdot}_{(x_0,f_0),\delta}$. 

    Suppose that $y\in B_{(x_0,f_0),\delta}$. Then $\|y\|\leq 1$. If $y\notin  C\left(f_0,1-\delta\right)$, we trivially have that $\vertiii{y}_{(x_0,f_0),\delta}=\|y\|\leq 1$. If $y\in  C^+(f_0,1-\delta)$, write $y= \lambda z + (1-\lambda)\left(1-\frac{\delta}{2}\right)x_0$ for some $z\in B_{\|\cdot\|}$ with $|\langle f_0,z\rangle|<(1-\delta)\|z\|$ and $\lambda\in [0,1]$. Using the properties of the map $\psi$ and the fact that $\psi(y)\leq\langle f_0,y\rangle$ and $\psi(z)\geq\langle f_0,z\rangle$, we obtain that
    $$ \lambda\leq\frac{1-\frac{\delta}{2}}{\left(1-\frac{\delta}{2}\right)+\psi(z)-\langle f_0,z\rangle}$$
    and thus we can write
    $$y =\mu z_1+(1-\mu)\left(1-\frac{\delta}{2}\right)x_0, $$
    where $0\leq \mu=\lambda\left(1+\frac{\psi(z)-\langle f_0,z\rangle}{1-\frac{\delta}{2}}\right)\leq 1$ and
    $$z_1=\frac{1-\frac{\delta}{2}}{\left(1-\frac{\delta}{2}\right)+\psi(z)-\langle f_0,z\rangle}(z+(\psi(z)-\langle f_0,z\rangle)x_0)\in B_{\|\cdot\|}. $$
    Notice that $\langle f_0,z_1\rangle =(1-\delta)\|z_1\|$ and thus $\psi(z_1)=\langle f_0,z_1\rangle$.  
    Then, using the definition of $\vertiii{y}_{(x_0,f_0),\delta}$ and the properties of $\psi$ we have that 
    \begin{align*}
        \vertiii{y}_{(x_0,f_0),\delta}&=1-\mu+\mu\left(\frac{\psi(z_1)}{1-\delta}+ \frac{\langle f_0,z_1\rangle-\psi(z_1)}{1-\frac{\delta}{2}}\right)\\ 
        &= 1-\mu+\mu\left(\frac{\psi(z_1)}{1-\delta}\right)\leq 1.
    \end{align*}
    A similar argument shows that $\vertiii{y}_{(x_0,f_0),\delta}\leq 1$ if $y\in C^-(f_0,1-\delta)$ as well. Hence $B_{(x_0,f_0),\delta}\subset\{y\in X\colon \vertiii{y}_{(x_0,f_0),\delta}\leq 1\}$.

    To show the converse inclusion, it suffices to prove that $y\in B_{(x_0,f_0),\delta}$ whenever $\vertiii{y}_{(x_0,f_0),\delta}=1$, again by virtue of the homogeneity of $\vertiii{\cdot}_{(x_0,f_0),\delta}$. Given such a vector $y\in X$, if $y\notin  C\left(f_0,1-\delta\right)$ there is nothing to prove. Suppose then first that $y\in C^+(f_0,1-\delta)$. We may write 
    $$ y = a\frac{1-\delta}{\psi(y)} \left(y+(\psi(y)-\langle f_0,y\rangle)x_0\right)+b\left(1-\frac{\delta}{2}\right)x_0,$$
    where $a= \frac{\psi(y)}{1-\delta}$ and $b=\frac{\langle f_0,y\rangle -\psi(y)}{1-\frac{\delta}{2}}$. Notice that $a$ and $b$ are positive numbers such that $a+b=\vertiii{y}_{(x_0,f_0),\delta}=1$, and that the vector $\frac{1-\delta}{\psi(y)} \left(y+(\psi(y)-\langle f_0,y\rangle)x_0\right)$ lies in the closure of the set $B_{\|\cdot\|}\setminus C\left(f_0,1-\delta\right)$. This shows that $y\in B_{(x_0,f_0),\delta}$. We may follow a similar process to show that $y\in B_{(x_0,f_0),\delta}$ whenever $\vertiii{y}_{(x_0,f_0),\delta}=1$ and $y\in C^-(f_0,1-\delta)$.

    It remains to show that $\vertiii{\cdot}_{(x_0,f_0),\delta}$ is an equivalent norm. We have already shown that its unit ball is contained in the ball $B_{\|\cdot\|}$. Hence, we will finish the proof by showing that $\vertiii{y}_{(x_0,f_0),\delta}\leq 1$ for all $y\in (1-\delta)B_{\|\cdot\|}$. If $y\in (1-\delta)B_{\|\cdot\|}$ and $y$ does not belong to $C(f_0,1-\delta)$ then $\vertiii{y}_{(x_0,f_0),\delta}\leq 1$ holds trivially. Otherwise, if $y\in C^+(f_0,1-\delta)$, using that $\langle f_0,y\rangle\geq\psi(y)$ we have that
    $$\vertiii{y}_{(x_0,f_0),\delta}\leq \frac{\langle f_0,y\rangle}{1-\delta}\leq 1. $$
    The same holds for $y\in C^-(f_0,1-\delta)$, and the result is proven.
\end{proof}

\subsection{Quantitative estimates in the norm $\vertiii{\cdot}_{(x_0,f_0),\delta}$}

First we obtain an estimate on the ``non-differentiability" of the norm $\vertiii{\cdot}_{(x_0,f_0),\delta}$ at the point $\left(1-\frac{\delta}{2}\right)x_0$.

\begin{lemma}
    \label{Lemma1_Derivative}
    For all $h\in S_{\|\cdot\|}$, the inequality
\begin{align*}
   \frac{\vertiii{\left(1-\frac{\delta}{2}\right)x_0+th}_{(x_0,f_0),\delta}-\vertiii{\left(1-\frac{\delta}{2}\right)x_0}_{(x_0,f_0),\delta}}{t}\geq \frac{\delta -(4-\delta)|\langle f_0,h\rangle|}{2\left(1-\frac{\delta}{2}\right)(2-\delta)}
\end{align*}
    holds for all $0<t<\frac{\delta\left(1-\frac{\delta}{2}\right)}{(1-\delta)+|\langle f_0,h\rangle|}$.
\end{lemma}

\begin{proof}

Fix $h\in S_{\|\cdot\|}$. First we obtain a lower bound on the value $\psi(h)$. Specifically, we get
    \begin{align*}
        \psi(h)&=(1-\delta)\|h+\psi(h)x_0-\langle f_0,h \rangle x_0\|\\
        &\geq (1-\delta)\left(1-\psi(h)-|\langle f_0,h\rangle|\right),
    \end{align*}
    from where we deduce that 
    \begin{equation*}
        \psi(h)\geq\frac{1-\delta}{2-\delta}\left(1-|\langle f_0,h\rangle|\right).
    \end{equation*}
    
    Let $0<t<\frac{\delta\left(1-\frac{\delta}{2}\right)}{(1-\delta)+|\langle f_0,h\rangle|}$. Then we have that $\left(1-\frac{\delta}{2}\right)x_0+th$ belongs to the cone $C^+(f_0,1-\delta)$. Hence:
    \begin{align*}
        \vertiii{\left(1-\frac{\delta}{2}\right)x_0+th}_{(x_0,f_0),\delta}&= 1+t\left(\frac{\delta \psi(h)}{2(1-\delta)\left(1-\frac{\delta}{2}\right)}+\frac{\langle f_0,h\rangle}{\left(1-\frac{\delta}{2}\right)}\right)\\
        &\geq 1+t \left(\frac{\delta(1-|\langle f_0,h\rangle|)-2(2-\delta)|\langle f_0,h\rangle|}{2(2-\delta)\left(1-\frac{\delta}{2}\right)}\right)\\
        &= 1+t\left(\frac{\delta-(4-\delta)|\langle f_0,h\rangle|}{2(2-\delta)\left(1-\frac{\delta}{2}\right)}\right).
    \end{align*}

    The desired inequality follows immediately.
\end{proof}

We can also get an upper bound on the diameter of the slices generated by the functional $f_0$ close to the point $\left(1-\frac{\delta}{2}\right)x_0$:

\begin{lemma}
    \label{Lemma2_Slices}
    For every $0<\varepsilon<\delta/2$ we have that

    $$\text{diam}_{\vertiii{\cdot}_{(x_0,f_0),\delta}}\left(S\left(B_{\vertiii{\cdot}_{(x_0,f_0),\delta}},f_0, 1-\frac{\delta}{2}-\varepsilon\right)\right)\leq \left(\frac{8-2\delta}{\delta}\right)\varepsilon.$$
\end{lemma}
\begin{proof}
    Fix $x\in S\left(B_{\vertiii{\cdot}_{(x_0,f_0),\delta}},f_0, 1-\frac{\delta}{2}-\varepsilon\right)$. Since  $\langle f_0, x\rangle > 1-\frac{\delta}{2}-\varepsilon>1-\delta$ and $x\in B_{(x_0,f_0),\delta}$, we necessarily have that $x\in C^+(f_0,1-\delta)$. Therefore, there exists $z\in B_{\|\cdot\|}$ with $|\langle f_0,z\rangle| <(1-\delta)\|z\|$ and $\lambda\in [0,1]$ such that
    $$ x=\lambda z+(1-\lambda)\left(1-\frac{\delta}{2}\right)x_0.$$

    Hence, we have that
    $$1-\frac{\delta}{2}-\varepsilon< \langle f_0,x\rangle \leq \lambda(1-\delta)+(1-\lambda)\left(1-\frac{\delta}{2}\right), $$
    from which we obtain the bound $\lambda<\frac{2}{\delta}\varepsilon $. Now, a simple computation yields:
    \begin{align*}
        \vertiii{x-\left(1-\frac{\delta}{2}\right)x_0}_{(x_0,f_0),\delta}\leq \lambda\left(2-\frac{\delta}{2}\right)\leq \left(\frac{4-\delta}{\delta}\right)\varepsilon.
    \end{align*}
    A straightforward use of the triangle inequality finishes the proof.
\end{proof}

The third estimate we obtain of the new norm $\vertiii{\cdot}_{(x_0,f_0),\delta}$ quantifies how much larger this norm is than the starting norm $\|\cdot\|$ around $\left(1-\frac{\delta}{2}\right)x_0$. The proof follows directly from the definition:

\begin{lemma}
\label{Lemma3_Norm_in_cone}
If $x\in C\left(f_0,1-\frac{\delta}{4}\right)$, then 
$$ \|x\|\leq \frac{4-2\delta}{4-\delta}\vertiii{x}_{(x_0,f_0),\delta}.$$
\end{lemma}
\begin{proof}
    Clearly, it is enough to prove the result for $x\in C^+(f_0,1-\frac{\delta}{4})$. In this case, we have that 
    \begin{align*}
        \vertiii{x}_{(x_0,f_0),\delta}&=\frac{2\psi(x)}{\delta(1-\delta)\left(1-\frac{\delta}{2}\right)}+\frac{\langle f_0,x\rangle}{\left(1-\frac{\delta}{2}\right)}\\
        &\geq \frac{4-\delta}{4-2\delta}\|x\|,
    \end{align*}
    because $\psi(x)$ is non-negative. 
\end{proof}

In $c_0(\Gamma)$, the canonical biorthogonal system $\{e_\gamma;e_\gamma\}_{\gamma\in\Gamma}$ produces norms with an additional property. We say that a norm $\vertiii{\cdot}$ in $c_0(\Gamma)$ is a \emph{lattice norm} if $\|x\|\leq \|y\|$ whenever $x,y\in c_0$ and $|x_\gamma|\leq |y_\gamma|$ for all $\gamma\in \Gamma$. Clearly the usual supremum norm on $c_0(\Gamma)$ is a lattice norm.

\begin{lemma}
    \label{Prop2_Lattice_norm}
    Let $\Gamma$ be a set and let $\gamma_0\in \Gamma$. The norm $\vertiii{\cdot}_{(e_{\gamma_0},e_{\gamma_0}),\delta}$ is a lattice norm.
\end{lemma}
\begin{proof}
    Notice that
    \begin{equation}
        \label{Eq1_Prop_lattice}
        \vertiii{x}_{(\gamma_0,\gamma_0),\delta}=\frac{\delta\psi(x)}{2(1-\delta)\left(1-\frac{\delta}{2}\right)}+\frac{x_{\gamma_0}}{\left(1-\frac{\delta}{2}\right)}
    \end{equation}
    for all $x\in c_0(\Gamma)$ with $x\in C^+(e_{\gamma_0},1-\delta)$.
    
    Let $x$ and $y$ be two points in $c_0(\Gamma)$ such that $|x_\gamma|\leq |y_\gamma|$ for all $\gamma\in \Gamma$. We may assume without loss of generality that both $x_{\gamma_0}$ and $y_{\gamma_0}$ are non-negative, since $\vertiii{x}_{(e_{\gamma_0},e_{\gamma_0}),\delta}=\vertiii{-x}_{(e_{\gamma_0},e_{\gamma_0}),\delta}$ for all $x\in c_0(\Gamma)$. We can assume as well that $x\in C^+(e_{\gamma_0},1-\delta)$, since otherwise $\vertiii{x}_{(\gamma_0,\gamma_0),\delta}=\|x\|_\infty\leq \|y\|_\infty\leq \vertiii{y}_{(\gamma_0,\gamma_0),\delta}$ and we are done. 

    Suppose then that $x\in C^+(e_{\gamma_0},1-\delta)$. Since $\psi(x)=(1-\delta)\|x+(\psi(x)-x_{\gamma_0})e_{\gamma_0}\|_\infty$ and the $\gamma_0$-th coordinate of $x+(\psi(x)-x_{\gamma_0})e_{\gamma_0}$ is exactly $\psi(x)$, it follows that 
    $$\|x+(\psi(x)-x_{\gamma_0})e_{\gamma_0}\|_\infty =\sup_{\gamma\neq \gamma_0} |x_\gamma|. $$
    The same argument can be applied to $y+(\psi(y)-y_{\gamma_0})e_{\gamma_0}$ to obtain that 
    $$\|y+(\psi(y)-y_{\gamma_0})e_{\gamma_0}\|_\infty = \sup_{\gamma\neq \gamma_0} |y_\gamma|.$$
    Therefore, $\psi(x)\leq \psi(y)$. If $y\in C^+(e_{\gamma_0},1-\delta)$, simply applying equation \eqref{Eq1_Prop_lattice} yields the result. Otherwise, if $y_{\gamma_0}<(1-\delta)\|y\|_\infty$, we have that 
    \begin{align*}
        \vertiii{x}_{(\gamma_0,\gamma_0),\delta}&=\frac{\delta\psi(x)}{2(1-\delta)\left(1-\frac{\delta}{2}\right)}+\frac{x_{\gamma_0}}{\left(1-\frac{\delta}{2}\right)}\\
        &< \frac{\frac{\delta}{2}\|y\|_\infty}{\left(1-\frac{\delta}{2}\right)}+\frac{(1-\delta)\|y\|_\infty}{\left(1-\frac{\delta}{2}\right)}= \|y\|_\infty= \vertiii{y}_{(\gamma_0,\gamma_0),\delta},
    \end{align*}
    which finishes the proof.

\end{proof}

\subsection{Approximations to the norm $\vertiii{\cdot}_{(x_0,f_0),\delta}$}

We finish this section by showing that a good enough approximation of the norm $\vertiii{\cdot}_{(x_0,f_0),\delta}$ also satisfies similar estimates to the ones appearing in Lemmas \ref{Lemma1_Derivative} and \ref{Lemma2_Slices}. The proof of the following result is a straightforward computation.

\begin{lemma}
    \label{Lemma4_Approx_estimates}
    Let $\eta>0$, and let $\vertiii{\cdot}_\eta$ be a norm on $X$ such that $\vertiii{\cdot}_\eta\leq \vertiii{\cdot}_{(x_0,f_0),\delta}\leq (1+\eta)\vertiii{\cdot}_\eta$. Then we have:
    \begin{itemize}
        \item[(i)] For all $h\in S_{\|\cdot\|}$, the inequality:
        $$ D\left(\vertiii{\cdot}_{\eta}, \left(1-\frac{\delta}{2}\right)x_0,h,t\right)\geq \left(\frac{2}{1+\eta}\right)\left(\frac{\delta -(4-\delta)|\langle f_0,h\rangle|}{2\left(1-\frac{\delta}{2}\right)(2-\delta)}-\frac{\eta}{|t|}\right)$$
        holds for all $0<|t|<\frac{\delta\left(1-\frac{\delta}{2}\right)}{(1-\delta)+|\langle f_0,h\rangle|}$.

        \item[(ii)] For all $\varepsilon>0$, 
        $$\text{diam}_{\vertiii{\cdot}_{(x_0,f_0),\delta}}\left(S\left(B_{\vertiii{\cdot}_{\eta}},f_0, 1-\frac{\delta}{2}-\varepsilon\right)\right)\leq \left(\frac{8-2\delta}{\delta}\right)\left(\varepsilon+\eta\left(1-\frac{\delta}{2}\right)\right).$$
    \end{itemize}
\end{lemma}
\begin{proof}
    For $(i)$, we obtain first the following inequality for a fixed $h\in S_{\|\cdot\|}$ and $0<t<\frac{\delta\left(1-\frac{\delta}{2}\right)}{(1-\delta)+|\langle f_0,h\rangle|}$ using Lemma \ref{Lemma1_Derivative}:
    \begin{align*}
        &\frac{\vertiii{\left(1-\frac{\delta}{2}\right)x_0+th}_{\eta}-\vertiii{\left(1-\frac{\delta}{2}\right)x_0}_{\eta}}{t}\\
        &\geq \frac{1}{1+\eta}\left(\frac{\vertiii{\left(1-\frac{\delta}{2}\right)x_0+th}_{(x_0,f_0),\delta}-\vertiii{\left(1-\frac{\delta}{2}\right)x_0}_{(x_0,f_0),\delta}-\eta}{t}\right)\\
        &\geq \frac{1}{1+\eta}\left(\frac{\delta -(4-\delta)|\langle f_0,h\rangle|}{2\left(1-\frac{\delta}{2}\right)(2-\delta)}-\frac{\eta}{t}\right).
    \end{align*}
    Since the previous inequality also holds for $-h$, item $(i)$ follows immediately.

    Part $(ii)$ follows from the fact that 
    $$ \frac{1}{1+\eta}S\left(B_{\vertiii{\cdot}_{\eta}},f_0, 1-\frac{\delta}{2}-\varepsilon\right)\subset S\left(B_{\vertiii{\cdot}_{(x_0,f_0),\delta}},f_0, 1-\frac{\delta}{2}-\frac{\varepsilon+\eta\left(1-\frac{\delta}{2}\right)}{1+\eta}\right)$$
    and Lemma \ref{Lemma2_Slices}.
\end{proof}

\section{Constructing the final renorming}

The following lemma allows us to combine a pair of smooth norms into one equally smooth norm that coincides with one of the initial ones at specific regions. No new ideas are introduced in our proof, which is a simple application of Corollary 1.1.23 in \cite{RuPhd} by Russo, where the reader can find the background for this type of smoothening techniques.  
\begin{lemma}
\label{Lemma5_Smooth_sup}
Let $X$ be a Banach space, let $k\in \mathbb{N}\cup\{\infty\}$, and let $\vertiii{\cdot}_1$ and $\vertiii{\cdot}_2$ be two equivalent $C^k$-smooth norms in $X$. Then, for every $\varepsilon>0$ there exists another equivalent $C^k$-smooth norm $\vertiii{\cdot}$ such that for all $x\in X$:

\begin{itemize}
    \item[(i)] $\vertiii{x}=\vertiii{x}_i$ if $\vertiii{x}_j\leq \frac{1}{1+\varepsilon}\vertiii{x}_i$, with $i\neq j\in\{1,2\}$.
    \item[(ii)] $\max\{\vertiii{x}_1,\vertiii{x}_2\}\leq \vertiii{x}\leq (1+\varepsilon)\max\{\vertiii{x}_1,\vertiii{x}_2\}$. 
\end{itemize}
Moreover, if $\vertiii{\cdot}_i$ is LFC for $i=1,2$, then we can choose $\vertiii{\cdot}$ to be LFC.
\end{lemma}
\begin{proof}
    Consider a $C^\infty$-smooth and convex function $\phi\colon \mathbb{R}\rightarrow \mathbb{R}$ such that $\phi(t)=0$ for $t\in [0,\frac{1}{1+\varepsilon}]$, $\phi(1)=1$ and $\phi(t)>1$ for $t\in(1,+\infty)$. The map $\nu\colon X\rightarrow \mathbb{R}$ defined by $\nu(x)=\phi(\vertiii{x}_{1})+\phi(\vertiii{x}_{2})$  is an even, convex and $C^k$-smooth function in $X$. Moreover, the set $B = \{x\in X\colon \nu(x)\leq 1\}$ is bounded. Therefore, by Corollary 1.1.23 in \cite{RuPhd} we have that the Minkowski functional associated to $B$ is a $C^k$-smooth equivalent norm, which we denote by $\vertiii{\cdot}$. It is enough to check that properties $(i)$ and $(ii)$ are verified for all $x\in X$ with $\vertiii{x}=1$. 

    Notice that if $\vertiii{\cdot}_i$ is LFC for $i=1,2$, then the function $\nu$ is LFC as well, and thus it follows from the same Corollary 1.1.23 in \cite{RuPhd} that $\vertiii{\cdot}$ is LFC.
\end{proof}

We can now prove the result that shows how to combine certain countably many smooth norms preserving the local geometry at specific points.

\begin{theorem}
\label{Theorem1_Main_construction}
Let $(X,\|\cdot\|)$ be a Banach space. Let $k\in\mathbb{N}\cup\{\infty\}$, and suppose that $\vertiii{\cdot}_0$ is an equivalent $C^k$-smooth in $X$. Let $\alpha>0$, and let $\{a_n\}_{n\in\mathbb{N}}$ and $\{b_n\}_{n\in\mathbb{N}}$ be two sequences such that $0<a_n<b_n<1$ for all $n\in\mathbb{N}$. Let $\{f_n\}_{n\in\mathbb{N}}$ be a weak$^*$ null sequence in $X^*$, and let $\{\eta_n\}_{n\in\mathbb{N}}$ be a sequence of strictly positive numbers converging to $0$ such that $ 1\leq \prod_{n\in\mathbb{N}} (1+\eta_n) <\infty$. Let $\{\vertiii{\cdot}_{\eta_n}\}_{n\in\mathbb{N}}$ be a sequence of equivalent $C^k$-smooth norms such that $\vertiii{\cdot}_{\eta_n}\leq \alpha \vertiii{\cdot}_0$ and
\begin{align*}
    \vertiii{x}_0\leq \frac{1}{1+\eta_n}\vertiii{x}_{\eta_n},\quad &\text{ for all }x\in C\left(f_n,a_n\right),\text{ and}\\
    \vertiii{x}_{\eta_n}\leq \frac{1}{1+\eta_n} \vertiii{x}_0,\quad &\text{ for all }x\in X\setminus\left( C\left(f_n,b_n\right)\right).
\end{align*} 
for all $n\in\mathbb{N}$. 

Then there exists an equivalent $C^k$-smooth norm $\vertiii{\cdot}$ in $X$ such that 
$$\sup_{n\in\mathbb{N}}\vertiii{\cdot}_{\eta_n}\leq \vertiii{\cdot}\leq \alpha\prod_{n\in\mathbb{N}} (1+\eta_n) \vertiii{\cdot}_0,$$
and such that for every $n\in\mathbb{N}$ and every $x\in C(f_n,a_n)$ with $x\notin \bigcup_{i\neq n}C(f_i,b_i)$ it holds that $\vertiii{x}=\vertiii{x}_{\eta_n}$.

If $\vertiii{\cdot}_0$ and $\vertiii{\cdot}_{\eta_n}$ are LFC for all $n\in\mathbb{N}$, we can take $\vertiii{\cdot}$ to be LFC.
\end{theorem}
\begin{proof}
We are going to inductively define $C^k$-smooth equivalent norms $\{\vertiii{\cdot}_n\}_{n=0}^{\infty}$ in $X$ using Lemma \ref{Lemma5_Smooth_sup} such that for all $n\in\mathbb{N}$ we have
\begin{itemize}
    \item[(i)] $\vertiii{\cdot}_{\eta_n}\leq \vertiii{\cdot}_n\leq \alpha\prod_{k\leq n}(1+\eta_k)\vertiii{\cdot}_0$,
    \item[(ii)] $\vertiii{\cdot}_{n-1}\leq \vertiii{\cdot}_{n}$, with 
    \begin{align*}
        \vertiii{x}_n &= \vertiii{x}_{\eta_n},\text{ for all }x\in C\left(f_n,a_n\right)\text{ such that }x\notin \bigcup_{i<n}C(f_i,b_i),\\
        \vertiii{x}_{n} &= \vertiii{x}_{n-1},\text{ for all }x\in X\setminus \left( C\left(f_{n},b_n\right)\right),        
    \end{align*}
    \item[(iii)] If $\vertiii{\cdot}_{n-1}$ and $\vertiii{\cdot}_{\eta_n}$ are LFC, then $\vertiii{\cdot}_{n}$ is LFC.
\end{itemize}

Suppose we have defined the norm $\vertiii{\cdot}_{n-1}$ for a fixed $n\in\mathbb{N}$. We apply Lemma \ref{Lemma5_Smooth_sup} to the norms $\vertiii{\cdot}_{n-1}$ and $\vertiii{\cdot}_{\eta_{n}}$ with $\varepsilon= \eta_{n}$ to obtain a $C^k$-smooth norm $\vertiii{\cdot}_{n}$ such that 
$$\max\{\vertiii{\cdot}_{\eta_{n}},\vertiii{\cdot}_{n-1}\}\leq \vertiii{\cdot}_{n}\leq \alpha\prod_{k\leq n}(1+\eta_k)\vertiii{\cdot}_0$$
and
\begin{align*}
    \vertiii{x}_{n} &= \vertiii{x}_{\eta_{n}},\quad \text{ for all }x\in X\text{ such that }\vertiii{x}_{n-1}\leq\frac{1}{1+\eta_{n}}\vertiii{x}_{\eta_{n}},\\
    \vertiii{x}_{n}&=  \vertiii{x}_{n-1},\quad \text{ for all }x\in X\text{ such that }\vertiii{x}_{\eta_{n}}\leq\frac{1}{1+\eta_{n}}\vertiii{x}_{n-1},
\end{align*} 
The first of the previous two equations implies that if a point $x\in X$ belongs to $C(f_n,a_n)$ but does not belong to $C(f_i,b_i)$ for all $i<n$, then $\vertiii{x}_n=\vertiii{x}_{\eta_n}$, since in this case item $(ii)$ implies that $\vertiii{x}_{i}=\vertiii{x}_0$ for all $i<n$ and by hypothesis we have that $\vertiii{x}_0\leq \frac{1}{1+\eta_n}\vertiii{x}_{\eta_n}$. From the second equation we obtain that $\vertiii{x}_n=\vertiii{x}_{n-1}$ for all $x\in X\setminus \left( C\left(f_{n},b_n\right)\right)$, using that $\vertiii{\cdot}_{0}\leq \vertiii{\cdot}_{n-1}$.

If both $\vertiii{\cdot}_{n-1}$ and $\vertiii{\cdot}_{\eta_{n}}$ are LFC, then the norm $\vertiii{\cdot}_n$ we obtain from Lemma \ref{Lemma5_Smooth_sup} is LFC as well.

Once the induction is done, define $\vertiii{\cdot}\colon X\rightarrow[0,\infty)$ by
$$\vertiii{x}=\max_{n\in\mathbb{N}}\vertiii{x}_n.$$ 
Since the sequence $\{f_n\}_{n\in\mathbb{N}}$ is weak$^*$ null, for every $x\in X$ there exists $n_x\in\mathbb{N}$ such that $|\langle f_n,x\rangle|<\frac{\|x\|}{8}$ for all $n\geq n_x$. Therefore, Fact \ref{Fact1_Neighbourhood_outside} and item $(ii)$ imply that there is a neighborhood $B_x$ around $x$ such that $\vertiii{y}_n=\vertiii{y}_{n_x}$ for all $n\geq n_x$ and all $y\in B_x$ . 

This means that $\vertiii{\cdot}$ is well defined and is a $C^k$-smooth norm in $X$. It also implies that if $\vertiii{\cdot}_0$ and $\vertiii{\cdot}_{n}$ are LFC for all $n\in\mathbb{N}$, then $\vertiii{\cdot}$ is LFC.

It is an equivalent norm as well since by item $(i)$ we have that 
$$\sup_{n\in\mathbb{N}}\vertiii{\cdot}_{\eta_n}\leq \vertiii{\cdot}\leq \alpha\prod_{n\in\mathbb{N}} (1+\eta_n)\vertiii{\cdot}_0.$$
Finally, it follows from $(ii)$ that $\vertiii{x}=\vertiii{x}_{\eta_n}$ for all $x\in C(f_n,a_n)$ such that $x\notin \bigcup_{i\neq n}C(f_i,b_i)$.

\end{proof}

We can finally proceed to the proofs of our two main results. 

The argument we follow in both proofs is very similar. The main difference between the separable and the $c_0(\Gamma)$ cases lies in the smooth approximation theorems we use. In the separable case, we use Theorem 2.10 in \cite{HajTal14} by Hájek and Talponen, which states that in a separable Banach space admitting a $C^k$-smooth norm, every norm can be approximated uniformly on bounded sets by a $C^k$-smooth norm. This is a powerful result which is the culmination of several partial solutions given in \cite{DevFonHaj96,DevFonHaj98}. Let us briefly mention as well that it is currently unknown whether this theorem holds in the non-separable setting, and this is the only obstacle preventing us generalizing Theorem \ref{Main_Theorem_A} to a larger class of Banach spaces.

We will discuss the result used for the $c_0(\Gamma)$ case before its proof below. 

\begin{proof}[Proof of Theorem \ref{Main_Theorem_A}]

    Let $k\in\mathbb{N}\cup\{\infty\}$, and let $(X,\|\cdot\|)$ be a separable Banach space such that $\|\cdot\|$ is a $C^k$-smooth norm. Fix $0<\delta<1/2$. Consider a decreasing sequence $\{\eta_n\}_{n\in\mathbb{N}}$ of strictly positive numbers converging to $0$ such that $ 1\leq \prod_{n\in\mathbb{N}} (1+\eta_n) \leq 2$ and such that $\eta_n< \sqrt[4]{\frac{4-\delta}{4-2\delta}}-1$ for all $n\in\mathbb{N}$.
    
    Let us prove the following general claim:
    \begin{claim}
        \label{Claim1_Biorthogonal}
        In a Banach space $(X,\|\cdot\|)$, given $\varepsilon>0$, there exists a sequence of pairs $\{(x_n,f_n)\}_{n\in\mathbb{N}}$ in $S_{\|\cdot\|}\times S_{\|\cdot\|^*}$ such that $(f_n)_{n\in\mathbb{N}}$ is weak$^*$ null, $\langle f_n,x_n\rangle =1$ for all $n\in\mathbb{N}$ and $\langle f_n,x_m\rangle <\varepsilon$ for all $n\neq m\in\mathbb{N}$.
    \end{claim}
    \begin{proof}
    By Josefson$-$Nissenzweig's Theorem, there exists a sequence $\{f_n\}_{n\in\mathbb{N}}$ in $S_{\|\cdot\|^*}$ which converges weak$^*$ to $0$. We are going to define inductively a sequence $\{y_k\}_{k\in\mathbb{N}}$ in $S_{\|\cdot\|}$ and a subsequence $\{f_{n_k}\}_{k\in\mathbb{N}}$ such that for all $k\in\mathbb{N}$:
    \begin{itemize}
        \item[(i)]  $\langle f_{n_k},y_k\rangle \geq 1-2^{-k}\varepsilon$,
        \item[(ii)] $\langle f_{n_j},y_i\rangle <\varepsilon/2$ for all $i,j\leq k$ with $i<j$. 
        \item[(iii)]$\langle f_{n_i},y_j\rangle =0$ for all $i,j\leq k$ with $i< j$.  
    \end{itemize} 
    Put $n_1=1$ and choose any $y_1$ in $S_{\|\cdot\|}$ such that $\langle f_1,y_1\rangle \geq 1-\varepsilon/2$.

    Suppose we have defined the first $k$ terms of such a sequence. Since $\{f_n\}_{n\in\mathbb{N}}$ is weak$^*$ null, we may find a natural number $m_1$ such that $\langle f_m, y_i\rangle <\varepsilon/2$ for all $i\leq k$ and $m\geq m_1$.
    
    On the other hand, for each $i\leq k$ we can define the linear projection $P_i\colon X\rightarrow \text{span}\{y_i\}$ given by $P_ix=\frac{\langle f_{n_i},x\rangle}{\langle f_{n_i},y_i\rangle} y_i$. Inductively, we define another finite sequence of linear maps $\{T_i\colon X\rightarrow \text{span}\{y_1,\dots,y_i\}\}_{i\leq k}$ by $T_1=P_1$ and $T_{i+1}=T_i+P_{i+1}\circ(I-T_i)$ for $1\leq i\leq k-1$. By property $(iii)$, it holds that $P_i\circ P_j =0$ for all $1\leq i<j\leq k$, and thus we can prove inductively that $T_i\circ P_j=0$ for all $1\leq i<j \leq  k$. Therefore, again inductively we obtain that for all $1\leq i\leq k$, the map $T_i$ is a linear projection onto the subspace $\text{span}\{y_1,\dots,y_i\}$. Moreover, yet another inductive argument shows that for all $i\leq k$ it holds that $P_i\circ T_k=P_i$, and consequently $\langle f_{n_i},(I-T_k)x\rangle =0$. 
    
    We now observe that there exists $m_2\in \mathbb{N}$ such that 
    $$ \sup_{x\in S_{\|\cdot\|}}\langle f_m, (I-T_k)x\rangle\geq 1-2^{-(k+2)}\varepsilon $$
    for all $m\geq m_2$. Indeed, otherwise we would be able to find an infinite sequence $\{z_m\}_{m\in\mathbb{N}}$ in the finite-dimensional subspace $\text{span}\{y_1,\dots,y_k\}$ such that $\{\langle f_m,z_m\rangle\}_{m\in\mathbb{N}}$ is bounded away from $0$, and a usual compactness argument yields a contradiction with the fact that $\{f_n\}_{n\in\mathbb{N}}$ is weak$^*$ null.

    Now, we define the next functional in the subsequence as $f_{n_{k+1}}$ with $n_{k+1}\geq \max \{m_1,m_2\}$. Property $(ii)$ of the induction is verified by the choice of $m_1$. On the other hand, we have selected $m_2$ such that we can find a point $y_{k+1}$ in the sphere $S_{\|\cdot\|}$ with $\langle f_{n_{k+1}},y_{k+1}\rangle\geq 1-2^{-(k+1)}\varepsilon$ and $\langle f_{n_i},y_{k+1}\rangle =0$ for all $i\leq k$. The induction is thus complete.

    To finish proving the claim, it suffices to apply Bishop$-$Phelps$-$Bollobás' Theorem (see e.g.: Page 376 in \cite{FabHabHajMonZiz11}) to the sequence $\{(f_{n_k},y_k)\}_{k\in\mathbb{N}}$.

    \end{proof}

Consider then a sequence of pairs $\{(x_n,f_n)\}_{n\in\mathbb{N}}$ in $S_{\|\cdot\|}\times S_{\|\cdot\|^*}$ such that $\{f_n\}_{n\in\mathbb{N}}$ is weak$^*$ null, $\langle f_n,x_n\rangle =1$ for all $n\in\mathbb{N}$ and $\langle f_n,x_m\rangle <\frac{\delta}{2(2-\delta)}$ for all $n\neq m\in\mathbb{N}$, as given by the claim.

For each $n\in\mathbb{N}$, denote by $\vertiii{\cdot}_{(x_n,f_n),\delta}$ the equivalent norm in $X$ given by Lemma \ref{Prop1_Description_norm} that we have studied in the previous section.

By Theorem 2.10 in \cite{HajTal14}, there exists a $C^k$-smooth norm $\vertiii{\cdot}_0$ such that
$$ \vertiii{\cdot}_0\leq (1+\eta_1)^2\|\cdot\|\leq (1+\eta_1)\vertiii{\cdot}_0.$$    
With the same theorem we can approximate as well each $\vertiii{\cdot}_{(x_n,f_n),\delta}$ by an equivalent $C^k$-smooth norm $\vertiii{\cdot}_{\eta_n}$ such that
$$\vertiii{\cdot}_{\eta_n}\leq \vertiii{\cdot}_{(x_n,f_n),\delta}\leq (1+\eta_n)\vertiii{\cdot}_{\eta_n} $$
for every $n\in\mathbb{N}$. 

Notice that by definition of $\vertiii{\cdot}_{(x_n,f_n),\delta}$, the choice of the sequence $\{\eta_n\}_{n\in\mathbb{N}}$, and Lemma \ref{Lemma3_Norm_in_cone}, we have the two following inequalities:
\begin{align*}
    \vertiii{x}_{\eta_n}\leq \frac{1}{1+\eta_n}\vertiii{x}_0,\quad &\text{ for all }x\in X\setminus\left( C\left(f_n,1-\delta\right)\right),\text{ and}\\
    \vertiii{x}_0\leq \frac{1}{1+\eta_n}\vertiii{x}_{\eta_n},\quad &\text{ for all }x\in C\left(f_n,1-\frac{\delta}{4}\right).
\end{align*}

We apply Theorem \ref{Theorem1_Main_construction} to $\vertiii{\cdot}_0$ and the sequence $\{\vertiii{\cdot}_{\eta_n}\}_{n\in\mathbb{N}}$ and obtain an equivalent $C^k$-smooth norm $\vertiii{\cdot}\colon X\rightarrow [0,+\infty)$ such that 
$$\sup_{n\in\mathbb{N}}\vertiii{\cdot}_{\eta_n}\leq \vertiii{\cdot}\leq 4\|\cdot\|$$
and such that for every $n\in\mathbb{N}$, if a point $x\in X$ belongs to $C\left(f_n,1-\frac{\delta}{4}\right)$ and is not in $C(f_i,1-\delta)$ for any $i\in\mathbb{N}$ with $i\neq n$, then $\vertiii{x}=\vertiii{x}_{\eta_n}$. Notice that by the choice of the sequence $\{x_n\}_{n\in\mathbb{N}}$, it follows that for every $n\in\mathbb{N}$ 
$$B_{\|\cdot\|}\left(\left(1-\frac{\delta}{2}\right)x_{n},\frac{\delta}{16}\right)\subset C\left(f_n,1-\frac{\delta}{4}\right)\setminus\left(\bigcup_{i\neq n}C(f_i,1-\delta)\right),$$
and thus the norm $\vertiii{\cdot}$ coincides with $\vertiii{\cdot}_{\eta_n}$ in the ball $B_{\|\cdot\|}\left(\left(1-\frac{\delta}{2}\right)x_{n},\frac{\delta}{16}\right)$.

Let us show that $\vertiii{\cdot}$ is not uniformly Gâteaux in any direction. Fix $h\in S_{\|\cdot\|}$, and put $\varepsilon_0=\frac{\delta}{16}$. We are going to show that for all $\tau>0$, we can find a natural number $n_0\in\mathbb{N}$ and $0<t_0<\tau$ such that 
$$D\left(\vertiii{\cdot},\left(1-\frac{\delta}{2}\right)x_{n_0}, h, t_0\right)>\varepsilon_0, $$
which implies that $\vertiii{\cdot}$ is not uniformly Gâteaux in the direction of $h$. 
Hence, fix $\tau>0$. Without loss of generality we may assume that $\tau<\frac{\delta}{16}$. Now, put $n_0\in\mathbb{N}$ such that $|\langle f_{n_0},h\rangle|<\frac{\delta}{2(4-\delta)}$ and such that $\frac{\eta_{n_0}}{1+\eta_{n_0}}\frac{32}{\delta}<\tau$. Then, if $t_0 = \frac{\eta_{n_0}}{1+\eta_{n_0}}\frac{32}{\delta}$, we have that $t_0<\tau<\frac{\delta}{16}$. 

Therefore, the point $\left(1-\frac{\delta}{2}\right)x_{n_0}\pm t_0h$ belongs to the ball $B_{\|\cdot\|}\left(\left(1-\frac{\delta}{2}\right)x_{n_0},\frac{\delta}{16}\right)$, which implies that
$$ D\left(\vertiii{\cdot},\left(1-\frac{\delta}{2}\right)x_{n_0}, h, t_0\right)=D\left(\vertiii{\cdot}_{\eta_{n_0}},\left(1-\frac{\delta}{2}\right)x_{n_0}, h, t_0\right).$$
Notice that $t_0<\frac{\delta\left(1-\frac{\delta}{2}\right)}{(1-\delta)+|\langle f_0,h\rangle|}$, so we can estimate the previous quotient using Lemma \ref{Lemma4_Approx_estimates}, and finally obtain:
\begin{align*}
    D\left(\vertiii{\cdot},\left(1-\frac{\delta}{2}\right)x_{n_0}, h, t_0\right)&\geq \left(\frac{2}{1+\eta_{n_0}}\right)\left(\frac{\delta-(4-\delta)|\langle f_0,h\rangle|}{2\left(1-\frac{\delta}{2}\right)(2-\delta)}-\frac{\eta_{n_0}}{t_0}\right)\\
    &\geq \frac{\delta}{4\left(1-\frac{\delta}{2}\right)(2-\delta)}-\frac{2\eta_{n_0}}{(1+\eta_{n_0})t_0}>\frac{\delta}{8}-\frac{\delta}{16}=\varepsilon_0.
\end{align*}
It only remains to show that unit ball of $\vertiii{\cdot}$ is dentable. Notice that since $\left(1-\frac{\delta}{2}\right)x_n$ is in the unit ball $B_{\vertiii{\cdot}}$ for all $n\in\mathbb{N}$, the slice $S\left(B_{\vertiii{\cdot}},f_n,1-\frac{\delta}{2}-\varepsilon\right)$ is nonempty for every $n\in\mathbb{N}$ and every $\varepsilon>0$. Moreover, for every $r>0$ we can find $\varepsilon_0$ and $n_0$ such that $\left(\frac{8-2\delta}{\delta}\right)\left(\varepsilon_0+\eta_{n_0}\left(1-\frac{\delta}{2}\right)\right)<r$, and thus by Lemma \ref{Lemma4_Approx_estimates} we obtain that the diameter of the associated slice $S\left(B_{\vertiii{\cdot}},f_{n_0},1-\frac{\delta}{2}-\varepsilon_0\right)$ in the equivalent norm $\vertiii{\cdot}_{(x_{n_0},f_{n_0}),\delta}$ is less than $r$. We conclude that $B_{\vertiii{\cdot}}$ is dentable.
\end{proof}

In the $c_0(\Gamma)$ application of our method, to obtain that the final renorming is additionally LFC, we will use Theorem 1 in \cite{FaHaZi97} by Fabián, Hájek and Zizler, who proved that every lattice norm in $c_0(\Gamma)$ can be approximated by a $C^\infty$-smooth LFC norm. Although the authors in the aforementioned article did not mention explicitly in their statement of the theorem that the approximating norm they obtain is LFC, this is indeed shown in their proof. Alternatively, we may refer to Theorem 1.1.26 in \cite{RuPhd}, where the exact statement that we need is given, along with a simpler proof of the result.

Notice as well that in the case of $c_0(\Gamma)$ we can work directly with the biorthogonal system given by the vectors in the canonical bases of $c_0$ and $\ell_1$, which means that we do not need Claim \ref{Claim1_Biorthogonal} to obtain an (almost) biorthogonal sequence; and also implies that the norms $\vertiii{\cdot}_{(e_\gamma,e_\gamma),\delta}$ are lattice norms for every $\gamma\in\Gamma$ as we proved in Lemma \ref{Prop2_Lattice_norm}.

\begin{proof}[Proof of Theorem \ref{Main_Theorem_B}]
    As mentioned, the proof of this theorem follows the same steps as the previous one, so we outline the whole argument but omit some explanations already given above.
    
    Write $\mathbb{N}\subset \Gamma$ for simplicity. Consider $0<\delta<1/2$, and consider a decreasing sequence $\{\eta_n\}_{n\in\mathbb{N}}$ of strictly positive numbers converging to $0$ such that $ 1\leq \prod_{n\in\mathbb{N}} (1+\eta_n) \leq 2$ and such that $\eta_n< \sqrt[4]{\frac{4-\delta}{4-2\delta}}-1$ for all $n\in\mathbb{N}$. 

    For every $n\in\mathbb{N}$, define $\vertiii{\cdot}_{(e_n,e_n),\delta}$ to be the norm given by Lemma \ref{Prop1_Description_norm} (obtained from the starting supremum norm $\|\cdot\|_\infty$). For every $n\in\mathbb{N}$, Lemma \ref{Prop2_Lattice_norm} indicates that the norm $\vertiii{\cdot}_{(e_n,e_n),\delta}$ is a lattice norm, and so by Theorem 1 in \cite{FaHaZi97} we can approximate it by an equivalent LFC and $C^\infty$-smooth norm $\vertiii{\cdot}_{\eta_n}$ such that
    $$\vertiii{\cdot}_{\eta_n}\leq \vertiii{\cdot}_{(e_n,e_n),\delta}\leq (1+\eta_n)\vertiii{\cdot}_{\eta_n}. $$

    Using again Theorem 1 in \cite{FaHaZi97}, consider in $(c_0(\Gamma),\|\cdot\|_\infty)$ an equivalent LFC and $C^\infty$-smooth norm, which we denote by $\vertiii{\cdot}_0$ such that 
    $$\vertiii{\cdot}_0\leq (1+\eta_1)^2\|\cdot\|_\infty\leq(1+\eta_1)\vertiii{\cdot}_0. $$

We can apply Theorem \ref{Theorem1_Main_construction} to $\vertiii{\cdot}_0$ and the sequence of norms $\{\vertiii{\cdot}_{\eta_n}\}_{n\in\mathbb{N}}$ to obtain an equivalent $C^\infty$-smooth and LFC norm $\vertiii{\cdot}\colon c_0(\Gamma)\rightarrow [0,\infty)$ such that 
$$\sup_{n\in\mathbb{N}}\vertiii{\cdot}_{\eta_n}\leq \vertiii{\cdot}\leq 4\vertiii{\cdot}_0,$$
with the property that for every $n\in\mathbb{N}$ and every $x\in C(e_n,1-\frac{\delta}{4})$ such that $x\notin \bigcup_{i\neq n}C(e_i,1-\delta)$ it holds that $\vertiii{x}=\vertiii{x}_{\eta_n}$. The proof that $\vertiii{\cdot}$ is not uniformly Gâteaux and that the unit ball $B_{\vertiii{\cdot}}$ is dentable is now the exact same as in the previous theorem. 
\end{proof} 

\section*{Acknowledgments}
This research has been supported by PAID-01-19, by grant PID2021-122126NB-C33 funded by MCIN/AEI/10.13039/501100011033 and by “ERDF A way of making Europe”. This research was also supported by SGS22/053/OHK3/1T/13.

The author wishes to thank Petr Hájek for many and invaluable conversations which led to the elaboration of this paper. The author is also deeply grateful to Antonio José Guirao and Vicente Montesinos for their careful reading of the preprint and their constructive remarks.
\printbibliography

\end{document}